\documentclass[10pt]{article}
\usepackage{amsfonts}
\usepackage{latexsym}
\usepackage{cite}
\usepackage{amsmath,amsfonts,latexsym,amssymb}
\usepackage[mathscr]{eucal}
\usepackage{cases}
\usepackage{amsthm}

\usepackage[bf,small]{caption2}
\usepackage{float,color}
\usepackage{graphicx}
\usepackage{amsmath}
\usepackage{amssymb}
\usepackage[all]{xy}

\newtheorem{theorem}{theorem}[section]
\newtheorem{thm}[theorem]{Theorem}
\newtheorem{lem}[theorem]{Lemma}

\newtheorem{prop}[theorem]{Proposition}

\newtheorem{ass}[theorem]{Assertion}
\newtheorem{nota}[theorem]{Notation}

\newtheorem{exmp}[theorem]{Example}

\newtheorem{rmk}[theorem]{Remark}

\begin{document}

\title{\textbf{Character varieties of odd classical pretzel knots}}
\author{\Large Haimiao Chen
\footnote{Email: {\em\small chenhm@math.pku.edu.cn}} \\
\normalsize \em{Beijing Technology and Business University, Beijing, China}}

\date{}
\maketitle

\begin{abstract}
  We determine the ${\rm SL}(2,\mathbb{C})$-character variety for each odd classical pretzel knot $P(2k_1+1,2k_2+1,2k_3+1)$, and present a method for computing its A-polynomial.

\end{abstract}

\section{Introduction}

For a link $K\subset S^{3}$, let $E_{K}=S^3-N(K)$, with $N(K)$ a tubular neighborhood.
The ${\rm SL}(2,\mathbb{C})$-{\it representation variety} of $K$ is the set $\mathcal{R}(K)$ of representations $\rho:\pi_1(E_{K})\to{\rm SL}(2,\mathbb{C})$, and the {\it character variety} of $K$ is
$$\mathcal{X}(K)=\{\chi(\rho)\colon\rho\in\mathcal{R}(K)\},$$
where the {\it character} $\chi(\rho):\pi_1(E_{K})\to\mathbb{C}$ sends each $x\in\pi_1(E_{K})$ to ${\rm tr}(\rho(x))$.

It is well-known that (see \cite{CS83}) both $\mathcal{R}(K)$ and $\mathcal{X}(K)$ can be defined by a finite set of polynomial equations, and this is why they are so named.
Actually, $\mathcal{X}(K)$ is the geometric-invariant-theoretic quotient of $\mathcal{R}(K)$ under the action of ${\rm SL}(2,\mathbb{C})$ via conjugation. To be more plain, denoting the subset of $\mathcal{R}(K)$ consisting of irreducible representations by $\mathcal{R}^{\rm irr}(K)$, it is also well-known that up to conjugacy, each $\rho\in\mathcal{R}^{\rm irr}(K)$ is determined by $\chi(\rho)$.
Call
$$\mathcal{X}^{\rm irr}(K)=\{\chi(\rho)\colon\rho\in\mathcal{R}^{\rm irr}(K)\}$$
the {\it irreducible character variety} of $K$. Our main concern is to study $\mathcal{X}^{\rm irr}(K)$; the characters of reducible representations are relatively easy to understand.

Since the seminal paper \cite{CS83}, much attention has been attracted to representations of 3-manifold groups into ${\rm SL}(2,\mathbb{C})$. One reason is that, representation/character variety encodes much topological and geometric information on the underlying 3-manifolds.

Among the invariants extracted from character variety, the {\it A-polynomial} (the definition is recalled in Section 4) proposed in \cite{CCGLS94} is rather important, in that, it plays a key role not only in the problem of exceptional fillings, but also in the celebrated {\it AJ Conjecture} which involves the behavior of quantum invariants.

Till now, character varieties have been found for the following links: torus knots \cite{Mu09}, double twist knots \cite{MPL11}, double twist links \cite{PT14}, $(-2,2m+1,2n)$-pretzel links \cite{Tr16}.
In this paper, we aim to determine $\mathcal{X}^{\rm irr}(K)$ when $K$ is an odd classical pretzel knot.
This is the first time to deal with a 3-parameter family of knots. We are faced with the complications caused by three non-commuting variables. However, we can still find a route to bypass the difficulty, reducing equalities between matrices to those of traces.

As the result shows, $\mathcal{X}^{\rm irr}(K)$ consists of a finite set of isolated points, several conics and an algebraic curve of high genus. Based on this, we present a method for explicitly computing the A-polynomial.

\medskip

{\bf Acknowledgement}

The author is supported by NSFC (Grant No. 11401014). He is grateful to Prof. Xun Yu at Tianjin University for beneficial conversations.

\section{Preliminary}

Let $\mathcal{M}(2,\mathbb{C})$ denote the set of $2\times 2$ matrices with entries in $\mathbb{C}$; it is a 4-dimensional vector space over $\mathbb{C}$. Let $I$ denote the $2\times 2$ identity matrix.

Given $t\in \mathbb{C}$ and $k\in\mathbb{Z}$, take $a$ with $a+a^{-1}=t$ and put
\begin{align}
\omega_{k}(t)=\begin{cases}
(a^{k}-a^{-k})/(a-a^{-1}), &a\notin\{\pm 1\}, \\
ka^{k-1}, &a\in\{\pm 1\};
\end{cases}
\end{align}
note that the right-hand-side is unchanged when $a$ is replaced by $a^{-1}$. It is easy to verify that for all $k\in\mathbb{Z}$,
\begin{align}
\omega_k(t)+\omega_{-k}(t)&=0, \\
\omega_{k+1}(t)-t\omega_{k}(t)+\omega_{k-1}(t)&=0, \\
\omega_{k}(t)^{2}-t\omega_{k}(t)\omega_{k-1}(t)+\omega_{k-1}(t)^{2}&=1.  \label{eq:omega}
\end{align}

If $X\in{\rm SL}(2,\mathbb{C})$ with ${\rm tr}(X)=t$, then by Cayley-Hamilton Theorem,
\begin{align}
X^{-1}=tI-X, \label{eq:X-inverse}
\end{align}
and inductively we can obtain that for all $k\in\mathbb{Z}$,
\begin{align}
X^{k}=\omega_{k}(t)X-\omega_{k-1}(t)I. \label{eq:X^k}
\end{align}

\begin{lem}  \label{lem:basic}
For any $X,Y\in{\rm SL}(2,\mathbb{C})$ with ${\rm tr}(X)=t_{1}$, ${\rm tr}(Y)=t_{2}$ and ${\rm tr}(XY)=t_{12}$, one has
\begin{align}
XYX&=t_{12}X-Y^{-1}, \label{eq:basic-1}  \\
XY+YX&=(t_{12}-t_{1}t_{2})I+t_{2}X+t_{1}Y.   \label{eq:basic-2}
\end{align}
\end{lem}

\begin{proof}
The first identity follows from $XYXY=t_{12}XY-I$, and the second one can be deduced from
$$X(t_2I-Y)+Y(t_1I-X)=XY^{-1}+YX^{-1}={\rm tr}(XY^{-1})\cdot I=(t_1t_2-t_{12})I.$$
\end{proof}

\begin{lem} \label{lem:3 matrix}
Given $t, t_{12}, t_{23}, t_{13}, t_{123}\in\mathbb{C}$, let
\begin{align*}
\nu_0&=t^2(3-t_{13}-t_{23}-t_{13})+t_{12}^2+t_{23}^2+t_{13}^2+t_{12}t_{23}t_{13}-4, \\
\nu_1&=t(t_{12}+t_{23}+t_{13})-t^3.
\end{align*}

{\rm(i)} There exist $X_1, X_2, X_3\in{\rm SL}(2,\mathbb{C})$ with ${\rm tr}(X_1)={\rm tr}(X_2)={\rm tr}(X_3)=t$,
${\rm tr}(X_{i}X_{j})=t_{ij}$, $1\le i<j\le 3$, and ${\rm tr}(X_1X_2X_3)=t_{123}$ if and only if
\begin{align}
t_{123}^2-\nu_1t_{123}+\nu_0=0.   \label{eq:r-gamma}
\end{align}

{\rm(ii)} If {\rm(\ref{eq:r-gamma})} holds and $X_1,X_2,X_3$ are required to have no common eigenvector, then the ordered triple $(X_1,X_2,X_3)$ is unique up to simultaneous conjugacy.
\end{lem}
Here (i) is a special case of the statement next to Lemma 2.3 in \cite{Lu99}, 
and (ii) follows from Lemma 2.4 in \cite{Lu99} (referred to Lemma 1.5.2 in \cite{CS83}), restating that up to conjugacy, an irreducible representation of the free group of rank 3 is determined by the tuple $(t,t_{12},t_{23},t_{13},t_{123})$ satisfying (\ref{eq:r-gamma}). Also see \cite{Ho72}.

\medskip

For $a,b\in\mathbb{C}$ with $a\ne 0$, put
\begin{align}
U(a,b)=\left(\begin{array}{cc} a & b \\ 0 & a^{-1} \end{array}\right), \qquad
V(a,b)=\left(\begin{array}{cc} a & 0 \\ b & a^{-1} \end{array}\right).
\end{align}

\section{The Character variety}

\subsection{Set up}

\begin{figure} [h]
  \centering
  \includegraphics[width=0.4\textwidth]{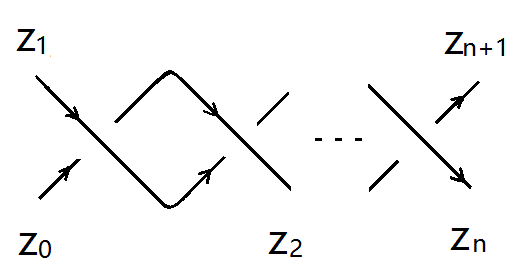}\\
  \caption{The tangle as a part of a link} \label{fig:n-tangle}
\end{figure}

For the Wirtinger presentation of a link $L$, refer to Theorem 3.4 in  \cite{BZ03}. Suppose a part of $L$ is an integral tangle as shown in Figure \ref{fig:n-tangle}. Let $z_0,z_1,\ldots,z_n$ denote the elements of $\pi_1(E_L)$ corresponding to the labeled directed arcs.
Since
$z_{k+1}z_{k}=z_{k}z_{k-1}=\cdots =z_{1}z_{0}$,
we have
\begin{align*}
z_{k+2}=z_{k+1}z_{k}z_{k+1}^{-1}=(z_{k+1}z_{k})z_{k}(z_{k+1}z_{k})^{-1}=(z_{1}z_{0})z_{k}(z_{1}z_{0})^{-1}.
\end{align*}
Hence for all $h\in\mathbb{Z}$,
\begin{align}
z_{2h}=(z_{1}z_{0})^{h}z_{0}(z_{1}z_{0})^{-h}, \qquad
z_{2h+1}=(z_{1}z_{0})^{h}z_{1}(z_{1}z_{0})^{-h}.  \label{eq:zn}
\end{align}

In this paper we consider the odd classical pretzel knot
\begin{align}
K=P(2k_1+1,2k_2+1,2k_3+1), \qquad k_1,k_2,k_3\in\mathbb{Z}.
\end{align}
A diagram for $K$ is shown in Figure \ref{fig:P3} when $k_1=1, k_2=k_2=2$.

\begin{figure}
  \centering
  \includegraphics[width=0.45\textwidth]{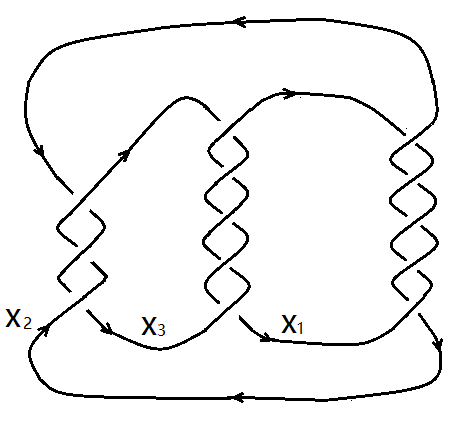}\\
  \caption{The pretzel knot $P(3,5,5)$} \label{fig:P3}
\end{figure}

Let $\rho:\pi_1(E_K)\to{\rm SL}(2,\mathbb{C})$ be a representation, and let $X_j=\rho(x_j), j=1,2,3$.
The following can be deduced using (\ref{eq:zn}):
\begin{align}
(X_{3}X_{1}^{-1})^{k_{2}+1}X_{1}(X_{3}X_{1}^{-1})^{-k_{2}-1}
&=(X_{2}X_{3}^{-1})^{k_1}X_{2}(X_{2}X_{3}^{-1})^{-k_1}, \label{eq:relation1} \\
(X_{1}X_{2}^{-1})^{k_{3}+1}X_{2}(X_{1}X_{2}^{-1})^{-k_{3}-1}
&=(X_{3}X_{1}^{-1})^{k_{2}}X_{3}(X_{3}X_{1}^{-1})^{-k_{2}}, \label{eq:relation2} \\
(X_{2}X_{3}^{-1})^{k_{1}+1}X_{3}(X_{2}X_{3}^{-1})^{-k_{1}-1}
&=(X_{1}X_{2}^{-1})^{k_3}X_{1}(X_{1}X_{2}^{-1})^{-k_{3}}. \label{eq:relation3}
\end{align}
Conversely, any $X_1,X_2,X_3\in{\rm SL}(2,\mathbb{C})$ satisfying these relations give rise to a representation.
\begin{rmk}
\rm
Since the $X_j$'s are conjugate to each other, none of them can be $\pm I$, unless $\rho$ is trivial. Let us assume $\rho$ to be nontrivial.
\end{rmk}

\begin{nota}
\rm To simplify the writing, by $X_{j+}$ we mean $X_{j+1}$ for $j\in\{1,2\}$ and $X_{1}$ for $j=3$; by $X_{j-}$ we mean $X_{j-1}$ for $j\in\{2,3\}$ and $X_{3}$ for $j=1$. Similarly for other situations.
\end{nota}

For $j=1,2,3$, put
\begin{align}
Y_{j}&=X_{j+}X_{j-}^{-1}, \\
A_{j}&=Y_{j}^{k_{j}}X_{j+}Y_{j}^{-k_{j}}X_{j-}.
\end{align}
The relations (\ref{eq:relation1})--(\ref{eq:relation3}) are equivalent to
\begin{align}
A_{1}=A_{2}=A_{3}.  \label{eq:A}
\end{align}

\begin{rmk} \label{rmk:reducible}
\rm
If $\rho$ is reducible, then taking conjugation if necessary, we may assume that the elements of the image of $\rho$ are all upper-triangular, and furthermore, $X_1=U(u,1)$ for some $u\ne 0$.
Since $X_1,X_2,X_3$ are conjugate to each other via upper-triangular elements, we have $X_j=U(u,a_j), j=2,3$ for some $a_j$; let $a_1=1$.
Then $Y_j=U(1,u(a_{j+}-a_{j-}))$ so that ${\rm tr}(Y_j)=2$, and
\begin{align*}
A_j&=U(1,k_ju(a_{j+}-a_{j-}))U(u,a_{j+})U(1,-k_ju(a_{j+}-a_{j-}))U(u,a_{j-}) \\
&=U(u^2,((k_j+1)u^{-1}-k_ju)a_{j+}+((k_j+1)u-k_ju^{-1})a_{j-}).
\end{align*}
The condition (\ref{eq:A}) is equivalent to
\begin{align*}
(k_1+k_3+1)(u^{-1}-u)(a_2-1)+((k_1+1)u-k_1u^{-1})(a_3-1)=0, \\
((k_1+1)u^{-1}-k_1u)(a_2-1)+(k_1+k_2+1)(u-u^{-1})(a_3-1)=0.
\end{align*}

All reducible representations can be found this way. Note that $\rho$ is abelian if and only if $a_2=a_3=1$.
\end{rmk}

\subsection{Irreducible representations}

In this subsection, $\rho$ is assumed to be irreducible; equivalently, $X_1,X_2,X_3$ have no common eigenvector.

Suppose
\begin{align}
{\rm tr}(X_{1}X_{2}X_{3})&=r; \qquad {\rm tr}(X_{1})={\rm tr}(X_{2})={\rm tr}(X_{3})=t=u+u^{-1}, \\
&{\rm tr}(Y_{j})=s_j=v_j+v_j^{-1}, \qquad  j=1,2,3.  \label{eq:trace}
\end{align}
Clearly,
\begin{align}
{\rm tr}(X_{j+}X_{j-})={\rm tr}(X_{j+}(tI-X_{j-}^{-1}))
=t^2-s_j.
\end{align}

Let
\begin{align}
\sigma_1=s_1+s_2+s_3,  \qquad \sigma_2&=s_1s_2+s_2s_3+s_3s_1,  \qquad \sigma_3=s_1s_2s_3, \\
\tau&=t^3+t-r, \\
\delta&=4+\sigma_3+2\sigma_2-\sigma_1^2, \\
\kappa&=\tau^2-t(\sigma_1+2)\tau+t^2(\sigma_2+4).
\end{align}
Then (\ref{eq:r-gamma}) can be rewritten as
\begin{align}
\kappa=\delta. \label{eq:t-sigma-mu}
\end{align}

\begin{rmk}
\rm Note that
\begin{align*}
\delta=-(s_3-v_1v_2-v_1^{-1}v_2^{-1})(s_3-v_1v_2^{-1}-v_1^{-1}v_2),
\end{align*}
so $\delta=0$ if and only if $v_3=v_1^{\epsilon_1}v_2^{\epsilon_2}$ with $\epsilon_1,\epsilon_2\in\{\pm 1\}$.
\end{rmk}

For each $j$, denote
\begin{align}
\alpha_j=\omega_{k_j-1}(s_j), \qquad \beta_{j}=\omega_{k_j}(s_{j}), \qquad \gamma_{j}=\omega_{k_{j}+1}(s_{j}).
\end{align}
Using (\ref{eq:X^k}) and (\ref{eq:basic-1}), we compute
\begin{align*}
Y_{j}^{k_{j}}X_{j+}&=(\beta_jY_j-\alpha_jI)X_{j+}=\beta_jX_{j+}X_{j-}^{-1}X_{j+}-\alpha_jX_{j+} \\
&=(s_j\beta_j-\alpha_j)X_{j+}-\beta_jX_{j-}=\gamma_{j}X_{j+}-\beta_{j}X_{j-}, \\
Y_{j}^{-k_{j}}X_{j-}&=(-\beta_jY_j+\gamma_jI)X_{j-}=\gamma_jX_{j-}-\beta_jX_{j+},
\end{align*}
so that
\begin{align}
A_{j}& =(\gamma_{j}X_{j+}-\beta_{j}X_{j-})(\gamma_{j}X_{j-}-\beta_{j}X_{j+})  \label{eq:Aj-0} \\
& =\gamma_{j}^{2}X_{j+}X_{j-}-\gamma_{j}\beta_{j}(X_{j+}^{2}+X_{j-}^{2})+\beta_{j}^{2}X_{j-}X_{j+},  \nonumber \\
& =(\gamma_j^2-\beta_j^2)X_{j+}X_{j-}+t(\beta_j-\gamma_j)\beta_j(X_{j+}+X_{j-})+(2\gamma_j-s_j\beta_j)\gamma_jI,  \label{eq:Aj-2}
\end{align}
where in the last line, (\ref{eq:basic-2}) is used.

\begin{lem} \label{lem:sj-ne2}
$s_j\ne 2$ for at least one $j$.
\end{lem}

\begin{proof}
Assume on the contrary that $s_j=2$ for all $j$.

If $u\notin\{\pm 1\}$, then taking conjugation if necessary, we may assume $X_1=U(u,0)$.
It follows from ${\rm tr}(X_1^{-1}X_2)=2$ that $X_2=U(u,a)$ or
$X_2=V(u,a)$ for some $a$. Similar situation occurs for $X_3$. Since $\rho$ is irreducible, we have
$\{X_2,X_3\}=\{U(u,a),V(u,b)\}$ for some $a,b\ne 0$. But then $s_1={\rm tr}(X_2X_3^{-1})=2-ab\ne 2$, a contradiction.

If $u\in\{\pm 1\}$, then taking conjugation if necessary, we may assume $X_1=U(u,u)$. It would follow from ${\rm tr}(X_1^{-1}X_2)={\rm tr}(X_1^{-1}X_3)=u^2+u^{-2}=2$ that $X_2,X_3$ are both upper-triangular, contradicting the irreducibility of $\rho$.
\end{proof}

\begin{rmk} \label{rmk:irreducible}
\rm Referred to Remark \ref{rmk:reducible}, the condition that $s_j\ne 2$ for at least one $j$ is also sufficient for $\rho$ to be irreducible.
\end{rmk}

Representations of $\pi_1(E_K)$ with $t=0$ were determined by the author in \cite{Ch16}; recall the result of Section 3 there,
stated in a different form:
\begin{prop} \label{prop:0}
Suppose $t=0$ and {\rm(\ref{eq:t-sigma-mu})} is satisfied. Then {\rm(\ref{eq:A})} holds if and only if one of the following cases occurs:
\begin{itemize}
  \item $\gamma_j=\beta_j$, $j=1,2,3$;
  \item $\gamma_j=-\beta_j$, $j=1,2,3$;
  \item $\delta=0$, and there exist $\theta_j\in\mathbb{R}, j=1,2,3$, such that $s_j=2\cos\theta_j$ and
        $\cos(k_1+1)\theta_1=\cos(k_2+1)\theta_2=\cos(k_3+1)\theta_3\ne\pm1$.
\end{itemize}
\end{prop}

\begin{lem} \label{lem:A=I}
For each $j$, $A_j=I$ if and only if $\gamma_j=\beta_j$ or $X_{j+}X_{j-}=I$.
\end{lem}
\begin{proof}
By (\ref{eq:Aj-0}), $A_j=I$ if and only if
$$(\gamma_jX_{j-}-\beta_jX_{j+})=(\gamma_jX_{j+}-\beta_jX_{j-})^{-1},$$
which, by (\ref{eq:X-inverse}), is equivalent to
$$(\gamma_jX_{j+}-\beta_jX_{j-})+(\gamma_jX_{j-}-\beta_jX_{j+})={\rm tr}(\gamma_jX_{j+}-\beta_jX_{j-})\cdot I=t(\gamma_j-\beta_j)I,$$
i.e., $(\gamma_j-\beta_j)(X_{j+}+X_{j-}-tI)=0$. Hence $A_j=I$ if and only if either $\gamma_j=\beta_j$ or $X_{j+}=tI-X_{j-}=X_{j-}^{-1}$.
\end{proof}

\begin{prop} \label{prop:2}
If $\gamma_\ell\ne\beta_\ell$ and $\gamma_{\ell'}=\beta_{\ell'}$ for some $\ell,\ell'$, then {\rm(\ref{eq:A})} holds if and only if
\begin{align}
\gamma_{\ell\pm}=\beta_{\ell\pm}, \qquad s_{\ell+}+s_{\ell-}=s_\ell+2=t^2\in(-\infty,0)\cup(0,4).  \label{eq:X2}
\end{align}
\end{prop}

\begin{proof}
By Lemma \ref{lem:A=I}, $A_{\ell'}=I$.

If (\ref{eq:A}) holds, then $A_j=I$ for all $j$. Since $\gamma_{\ell}\ne \beta_{\ell}$, by Lemma \ref{lem:A=I} we have $X_{\ell+}X_{\ell-}=I$, implying $s_\ell=t^2-2$.
Let $\{\ell_0\}=\{1,2,3\}-\{\ell,\ell'\}$.
Then $X_{\ell_0+}X_{\ell_0-}\ne I$, as otherwise $X_1,X_2,X_3$ would have a common eigenvector, so $\gamma_{\ell_0}=\beta_{\ell_0}$.
Thus $\gamma_{\ell\pm}=\beta_{\ell\pm}$, which is equivalent to $v_{\ell\pm}^{2k_{\ell\pm}+1}=-1$; in particular, $|v_{\ell\pm}|=1$, so that $s_{\ell\pm}=2{\rm Re}(v_{\ell\pm})<2$.
Moreover,
$$t^2={\rm tr}((X_{\ell+}+X_{\ell+}^{-1})X_\ell)={\rm tr}(X_{\ell-}^{-1}X_\ell)+{\rm tr}(X_{\ell+}^{-1}X_\ell)=s_{\ell+}+s_{\ell-}<4.$$
If $t=0$, then $v_{\ell+}=-v_{\ell-}$ or $v_{\ell+}=-v_{\ell-}^{-1}$, either implying
$v_{\ell+}^{(2k_{\ell+}+1)(2k_{\ell-}+1)}=-v_{\ell-}^{(2k_{\ell+}+1)(2k_{\ell-}+1)}$, which contradicts $v_{\ell\pm}^{2k_{\ell\pm}+1}=-1$ as just obtained.

If (\ref{eq:X2}) is satisfied, then by Lemma \ref{lem:A=I}, $A_{\ell\pm}=I$.
Taking conjugation if necessary, we may assume $X_{\ell+}=U(u,0)$. Assume
$$X_{\ell-}=\left(\begin{array}{cc} a & b \\ c & t-a \\ \end{array}\right), \qquad
X_{\ell}=\left(\begin{array}{cc} a' & b' \\ c' & t-a' \\ \end{array}\right).$$
From ${\rm tr}(X_{\ell-}X_{\ell+}^{-1})=s_{\ell}=t^2-2$ we obtain $u^{-1}a+u(t-a)=t^2-2$, implying $a=u^{-1}$, hence $bc=0$. From ${\rm tr}(X_{\ell}X_{\ell+}^{-1})=s_{\ell-}\ne 2$ we obtain $b'c'\ne 0$. It follows from ${\rm tr}(X_\ell^{-1}(X_{\ell+}+X_{\ell-}))=s_{\ell+}+s_{\ell-}=t^2$ that $b'c+bc'=0$. Thus $b=c=0$, and $X_{\ell-}=U(u^{-1},0)=X_{\ell+}^{-1}$, establishing (\ref{eq:A}).
\end{proof}

\begin{prop} \label{prop:3}
Suppose $t\ne 0$, $\gamma_j\ne \beta_j$ for all $j$, and {\rm(\ref{eq:t-sigma-mu})} is satisfied. Let $\lambda=\tau/t$. Then {\rm(\ref{eq:A})} holds if and only if
\begin{align}
(\lambda-2-s_j)\gamma_j&=(\sigma_1-s_{j}-\lambda)\beta_j, \qquad j=1,2,3. \label{eq:equiv-main}
\end{align}
\end{prop}

\begin{proof}
Using (\ref{eq:Aj-2}) and $\gamma_j^2-s_j\gamma_j\beta_j+\beta_j^2=1$
which is a special case of (\ref{eq:omega}), the following can be computed directly:
\begin{align}
{\rm tr}(A_{j})& =(t^2-s_j)(\gamma_{j}^{2}-\beta_{j}^{2})+2t^2(\beta_j-\gamma_j)\beta_j+2(2\gamma_{j}-s_j\beta_j)\beta_{j} \nonumber \\
& =2-(s_j+2-t^2)(\gamma_{j}-\beta_{j})^2, \label{eq:tr-Aj} \\
{\rm tr}(A_jX^{-1}_{j+})&={\rm tr}(A_jX^{-1}_{j-})=t, \label{eq:tr=t} \\
{\rm tr}(A_jX_j^{-1})&=(\tau-t(1+s_j))\gamma_j^2+t(s_j+2-\sigma_1)\gamma_j\beta_j+(t(\sigma_1-s_{j}+1)-\tau)\beta_j^2 \nonumber \\
&=t(\gamma_j-\beta_j)((\lambda-2-s_j)\gamma_j-(\sigma_1-s_{j}-\lambda)\beta_j)+t. \label{eq:tr=balabala}
\end{align}

If {\rm(\ref{eq:A})} holds, then
\begin{align}
{\rm tr}(A_1)={\rm tr}(A_2)&={\rm tr}(A_3), \label{eq:equiv1} \\
{\rm tr}(A_1X_j^{-1})={\rm tr}(A_2X_j^{-1})&={\rm tr}(A_3X_j^{-1}),  \qquad j=1,2,3;  \label{eq:equiv2}
\end{align}
due to $t\ne 0$ and $\gamma_j\ne \beta_j$, from (\ref{eq:tr=t}) and (\ref{eq:tr=balabala}) we see that (\ref{eq:equiv2}) is equivalent to (\ref{eq:equiv-main}).

Now suppose (\ref{eq:equiv-main}) holds (so that (\ref{eq:equiv2}) is satisfied).
Then for $j=1,2,3$,
\begin{align}
(\sigma_1+2-2\lambda)\gamma_j&=(\sigma_1-s_j-\lambda)(\gamma_j-\beta_j), \label{eq:deduce1} \\
(\sigma_1+2-2\lambda)\beta_j&=(\lambda-2-s_j)(\gamma_j-\beta_j), \label{eq:deduce2}
\end{align}
hence
\begin{align}
&(\sigma_1+2-2\lambda)^2=(\sigma_1+2-2\lambda)^2(\gamma_j^2-s_j\gamma_j\beta_j+\beta_j^2) \nonumber \\
=\ &((\sigma_1-s_j-\lambda)^2-s_j(\sigma_1-s_j-\lambda)(\lambda-2-s_j)+(\lambda-2-s_j)^2)(\gamma_j-\beta_j)^2 \nonumber \\
=\ &(t^{-2}\kappa(s_j+2)-\delta)(\gamma_j-\beta_j)^2, \label{eq:deduce3}
\end{align}
where in the last line we use $s_j^3-\sigma_1s_j^2+\sigma_2s_j-\sigma_3=0$.
By (\ref{eq:t-sigma-mu}),
\begin{align}
(\sigma_1+2-2\lambda)^2=t^{-2}\kappa(s_j+2-t^2)(\gamma_j-\beta_j)^2.  \label{eq:deduce4}
\end{align}
If $\sigma_1+2-2\lambda=0$, then (\ref{eq:deduce1}) and (\ref{eq:deduce2}) imply $s_1=s_2=s_3=2$, contradicting Lemma \ref{lem:sj-ne2}.
Hence $\sigma_1+2-2\lambda\ne 0$, so that $\kappa\ne 0$ and
\begin{align}
s_j\ne t^2-2, \qquad j=1,2,3. \label{eq:sj-ne}
\end{align}
Consequently,
\begin{align*}
{\rm tr}(A_j)-2=-(s_j+2-t^2)(\gamma_j-\beta_j)^2=-t^2\kappa^{-1}(\sigma_1+2-2\lambda)^2,
\end{align*}
which is independent of $j$. Thus (\ref{eq:equiv1}) is established.

The proof is completed once $I, X_1^{-1}, X_2^{-1}, X_3^{-1}$ are shown to form a basis of $\mathcal{M}(2,\mathbb{C})$; then (\ref{eq:equiv1}), (\ref{eq:equiv2}) will imply that ${\rm tr}(A_1Z)={\rm tr}(A_2Z)={\rm tr}(A_3Z)$ for all $Z\in\mathcal{M}(2,\mathbb{C})$, forcing $A_1=A_2=A_3$.

Assume on the contrary that $I, X_1^{-1}, X_2^{-1}, X_3^{-1}$ are linearly dependent, so are $I, X_1, X_2, X_3$.
Without loss of generality, assume $X_3=aI+bX_1+cX_2$.
Then
\begin{align}
t&={\rm tr}(X_3)=2a+t(b+c), \label{eq:tr-1} \\
s_1&={\rm tr}(X_3X_2^{-1})=ta+s_3b+2c, \\
s_2&={\rm tr}(X_3X_1^{-1})=ta+2b+s_3c, \\
2&={\rm tr}(X_3X_3^{-1})=ta+s_2b+s_1c. \label{eq:tr-4}
\end{align}
The first equation implies
\begin{align}
\overline{X}_3=b\overline{X}_1+c\overline{X}_2, \qquad \text{with} \qquad \overline{X}_j=X_j-\frac{t}{2}I.  \label{eq:X-bar}
\end{align}
The following can be computed using (\ref{eq:basic-2}):
\begin{align}
\overline{X}_1\overline{X}_2+\overline{X}_2\overline{X}_1&=pI,  \qquad \text{with} \qquad p=\frac{1}{2}t^2-s_3, \\
\overline{X}^2_j&=dI, \qquad \text{with}  \qquad d=\frac{1}{4}t^2-1.
\end{align}
Then (\ref{eq:Aj-2}) can be re-written as
\begin{align*}
A_1&=(\gamma_1^2-\beta_1^2)\overline{X}_{2}\overline{X}_{3}+\frac{t}{2}(\gamma_1-\beta_1)^2(\overline{X}_{2}+\overline{X}_{3})+(\cdots)I \\
&=(\cdots)\overline{X}_1\overline{X}_2+\frac{t}{2}(\gamma_1-\beta_1)^2(b\overline{X}_1+(c+1)\overline{X}_2)+(\cdots)I, \\
A_2&=(\gamma_2^2-\beta_2^2)\overline{X}_{3}\overline{X}_{1}+\frac{t}{2}(\gamma_2-\beta_2)^2(\overline{X}_{3}+\overline{X}_{1})+(\cdots)I \\
&=(\cdots)\overline{X}_1\overline{X}_2+\frac{t}{2}(\gamma_2-\beta_2)^2((b+1)\overline{X}_1+c\overline{X}_2)+(\cdots)I, \\
A_3&=(\cdots)\overline{X}_{1}\overline{X}_{2}+\frac{t}{2}(\gamma_3-\beta_3)^2(\overline{X}_{1}+\overline{X}_{2})+(\cdots)I,
\end{align*}
where the $(\cdots)$'s stand for coefficients that are irrelevant.

By (\ref{eq:equiv1}), (\ref{eq:equiv2}), ${\rm tr}(A_1\overline{X}_\ell)={\rm tr}(A_2\overline{X}_\ell)={\rm tr}(A_3\overline{X}_\ell)$,
$\ell=1,2$. Consequently,
$${\rm tr}(\overline{X}_1\overline{X}_2)=p, \qquad {\rm tr}(\overline{X}_1^2)={\rm tr}(\overline{X}_2^2)=2d, \qquad
{\rm tr}(\overline{X}_1^2\overline{X}_2)={\rm tr}(\overline{X}_1\overline{X}_2^2)=0.$$
Comparing ${\rm tr}(A_j\overline{X}_1)$ and ${\rm tr}(A_j\overline{X}_2)$ for $j=1,2,3$, we respectively obtain
\begin{align*}
(\gamma_1-\beta_1)^2(2db+t_3(c+1))=(\gamma_2-\beta_2)^2(2d(b+1)+pc)&=(\gamma_3-\beta_3)^2(2d+p), \\
(\gamma_1-\beta_1)^2(pb+2d(c+1))=(\gamma_2-\beta_2)^2(p(b+1)+2dc)&=(\gamma_3-\beta_3)^2(2d+p),
\end{align*}
which lead to
\begin{align*}
2db+p(c+1)&=pb+2d(c+1), \\
2d(b+1)+pc&=p(b+1)+2dc.
\end{align*}
These force $p=2d$, so that $s_3=2$. With (\ref{eq:sj-ne}), one can deduce from (\ref{eq:tr-1})--(\ref{eq:tr-4}) that $s_1=s_2=2$, contradicting Lemma \ref{lem:sj-ne2}.
\end{proof}

\begin{rmk} \label{rmk:implication}
\rm It is worth highlighting some points contained in the proof. With (\ref{eq:equiv-main}) satisfied, we have the following implications:
\begin{enumerate}
  \item[\rm(i)] If $\sigma_1+2-2\lambda\ne 0$ and (\ref{eq:t-sigma-mu}) holds, then $\kappa\ne 0$, $\delta\ne 0$ and $\gamma_j\ne\beta_j$, $j=1,2,3$, as seen from (\ref{eq:deduce4}); also, $s_j\ne 2$ for some $j$, as seen from (\ref{eq:deduce1}) and (\ref{eq:deduce2}).
  \item[\rm(ii)] (\ref{eq:t-sigma-mu}) implies (\ref{eq:equiv1}) which reads
        \begin{align}
        (s_j+2-t^2)(\gamma_j-\beta_j)^2=(s_1+2-t^2)(\gamma_{1}-\beta_{1})^2, \qquad j=2,3; \label{eq:tr=}
        \end{align}
        conversely, if $\sigma_1+2-2\lambda\ne 0$ and $s_{\ell}\ne s_{\ell'}$ for some $\ell,\ell'$, then (\ref{eq:t-sigma-mu}) can follow from (\ref{eq:tr=}), since by (\ref{eq:deduce3}),
        $$\frac{\kappa(s_\ell+2)-t^2\delta}{\kappa(s_{\ell'}+2)-t^2\delta}=
        \frac{(\gamma_{\ell'}-\beta_{\ell'})^2}{(\gamma_{\ell}-\beta_{\ell})^2}=\frac{s_\ell+2-t^2}{s_{\ell'}+2-t^2}.$$
\end{enumerate}
\end{rmk}

\subsection{The character variety}

The results of Proposition \ref{prop:0}, Lemma \ref{lem:A=I}, Proposition \ref{prop:2} and Proposition \ref{prop:3} are summarized as:
\begin{thm}
The irreducible character variety of $K$ can be embedded in $\{(t,s_1,s_2,s_3,\tau)\in\mathbb{C}^5\}$ and is the disjoint union of four parts:
$$\mathcal{X}^{\rm irr}(K)=\mathcal{X}_0\sqcup\mathcal{X}_1\sqcup\mathcal{X}_2\sqcup\mathcal{X}_3,$$
where
\begin{itemize}
  \item $\mathcal{X}_0=\mathcal{X}_{0,1}\sqcup\mathcal{X}_{0,2}$, where $\mathcal{X}_{0,1}$ consists of $(0,s_1,s_2,s_3,\tau)$ with
        $$\tau^2=\delta\ne 0,\qquad \gamma_j=-\beta_j, \qquad j=1,2,3,$$
        and $\mathcal{X}_{0,2}$ consists of $(0,2\cos\theta_1,2\cos\theta_2,2\cos\theta_3,0)$ with
        $$\cos(2k_1+1)\theta_1=\cos(2k_2+1)\theta_2=\cos(2k_3+1)\theta_3\ne -1;$$
  \item $\mathcal{X}_1=\mathcal{X}_{1,1}\sqcup\mathcal{X}_{1,2}\sqcup\mathcal{X}_{1,3}$, where $\mathcal{X}_{1,\ell}$ consists of
        $(t,s_1,s_2,s_3,\tau)$ with
        $$\gamma_{\ell\pm}=\beta_{\ell\pm}, \qquad  t^2=s_{\ell}+2=s_{\ell+}+s_{\ell-}\in(-\infty,0)\cup(0,4);$$
  \item $\mathcal{X}_2$ consists of $(t,s_1,s_2,s_3,\tau)$ with
        \begin{align}
        s_j\in\left\{2\cos\left(\frac{(2h+1)\pi}{2k_j+1}\right)\colon h=0,\ldots,k_j\right\}, \\
        \tau^2-t(\sigma_1+2)\tau+t^2(\sigma_2+4)=\delta,  \label{eq:chi1}
        \end{align}
        so $\mathcal{X}_2$ has $(k_1+1)(k_2+1)(k_3+1)$ components, each being a conic;
  \item $\mathcal{X}_3$ consists of $(t,s_1,s_2,s_3,t\lambda)$ with
        \begin{align}
        t\ne 0, \qquad \sigma_1+2-2\lambda\ne 0, \nonumber \\
        t^2(\lambda^2-(\sigma_1+2)\lambda+(\sigma_2+4))=\delta,  \label{eq:last} \\
        (\lambda-2-s_j)\gamma_j=(\sigma_1-s_{j}-\lambda)\beta_j, \qquad j=1,2,3. \label{eq:equiv-main2}
        \end{align}
\end{itemize}
The dimensions are: $\dim\mathcal{X}_{0}=\dim\mathcal{X}_{1}=0$, $\dim\mathcal{X}_2=\dim\mathcal{X}_3=1$.
\end{thm}

Some supplements are in order:
\begin{itemize}
  \item $\mathcal{X}_0\subset\partial\mathcal{X}_3$; actually, we could have defined $\mathcal{X}_0\cup\mathcal{X}_3$ by replacing (\ref{eq:last}) and (\ref{eq:equiv-main2}) respectively  with (\ref{eq:chi1}) and
      \begin{align*}
      (\tau-t(2+s_j))\gamma_j=(t(\sigma_1-s_{j})-\tau)\beta_j, \qquad j=1,2,3,
      \end{align*}
      and adding (\ref{eq:tr=}) which, by Remark \ref{rmk:implication} (ii), is redundant for $\mathcal{X}_3$.
  \item Each of the points in $\mathcal{X}_1$ is isolated.
  \item $\mathcal{X}_2\cap\partial\mathcal{X}_3$ consists of $(t,s_1,s_2,s_3,\tau)$ with
        $$\tau=t(1+\frac{\sigma_1}{2}), \qquad \gamma_j=\beta_j,\qquad j=1,2,3.$$
\end{itemize}

\medskip

We investigate $\mathcal{X}_3$ in more detail. Recall Remark \ref{rmk:implication} (i) that $\delta\ne 0$.

Let $\ell\in\{1,2,3\}$. By (\ref{eq:equiv-main2}) for $j=\ell\pm$,
\begin{align}
(\lambda-2-s_{\ell+})\gamma_{\ell+}\beta_{\ell-}=(s_{\ell-}+s_\ell-\lambda)\beta_{\ell+}\beta_{\ell-}, \label{eq:condition1} \\
(\lambda-2-s_{\ell-})\gamma_{\ell-}\beta_{\ell+}=(s_{\ell+}+s_\ell-\lambda)\beta_{\ell+}\beta_{\ell-}. \label{eq:condition2}
\end{align}

\begin{ass}
On $\mathcal{X}_3$, $\gamma_{\ell-}\beta_{\ell+}=\gamma_{\ell+}\beta_{\ell-}$ if and only if
$s_{\ell+}=s_{\ell-}$ and $v_{\ell+}^{2(k_{\ell+}-k_{\ell-})}=1$.
\end{ass}
\begin{proof}
The ``if" part is trivial. Suppose $\gamma_{\ell-}\beta_{\ell+}=\gamma_{\ell+}\beta_{\ell-}$.
Then $(s_{\ell-}-s_{\ell+})(\gamma_{\ell+}-\beta_{\ell+})\beta_{\ell-}=0$. If $\beta_{\ell-}\ne 0$, then since $\gamma_{\ell+}\ne\beta_{\ell+}$, we have $s_{\ell+}=s_{\ell-}$. If $\beta_{\ell-}=0$, then $\gamma_{\ell-}\beta_{\ell+}=0$, which, due to $\gamma_{\ell-}\ne\beta_{\ell-}$, implies $\beta_{\ell+}=0$; by (\ref{eq:equiv-main2}) for $j=\ell\pm$, $\lambda=2+s_{\ell+}=2+s_{\ell-}$, hence also $s_{\ell+}=s_{\ell-}$. Thus $v_{\ell+}=v_{\ell-}$ or $v_{\ell+}=v_{\ell-}^{-1}$. Then $\gamma_{\ell-}\beta_{\ell+}=\gamma_{\ell+}\beta_{\ell-}$ becomes $v_{\ell+}^{2(k_{\ell+}-k_{\ell-})}=1$.
\end{proof}

Decompose $\mathcal{X}_3$ as the ``singular" part and the ``regular" part: let
\begin{align}
\mathcal{X}_3^{\rm sin}&=\{(t,s_1,s_2,s_3,t\lambda)\in\mathcal{X}_3\colon\gamma_{\ell-}\beta_{\ell+}=\gamma_{\ell+}\beta_{\ell-}\ \text{for\ all\ }\ell\} \\
&=\{(t,v+v^{-1},v+v^{-1},v+v^{-1},t\lambda)\in\mathcal{X}_3\colon v^{2k_1}=v^{2k_2}=v^{2k_3}\}, \\
\mathcal{X}_3^{\rm reg}&=\mathcal{X}_3-\mathcal{X}_3^{\rm ex}.
\end{align}
Obviously, $\mathcal{X}_3^{\rm sin}$ consists of finitely many points, unless $k_1=k_2=k_3$, in which case $\mathcal{X}_3^{\rm sin}$ is a curve with two components, each being parameterized by $v+v^{-1}$.

Let $(t,s_1,s_2,s_3,t\lambda)\in\mathcal{X}_3^{\rm reg}$, then $\gamma_{\ell-}\beta_{\ell+}\ne \gamma_{\ell+}\beta_{\ell-}$ for some $\ell$.
From (\ref{eq:condition1}), (\ref{eq:condition2}) we obtain
\begin{align}
\lambda&=\frac{s_{\ell-}(\gamma_{\ell-}\beta_{\ell+}-\beta_{\ell+}\beta_{\ell-})-
s_{\ell+}(\gamma_{\ell+}\beta_{\ell-}-\beta_{\ell+}\beta_{\ell-})}{\gamma_{\ell-}\beta_{\ell+}-\gamma_{\ell+}\beta_{\ell-}}+2,   \label{eq:lambda} \\
s_\ell&=\frac{(s_{\ell-}-s_{\ell+})(\gamma_{\ell+}\gamma_{\ell-}-\beta_{\ell+}\beta_{\ell-})}
{\gamma_{\ell-}\beta_{\ell+}-\gamma_{\ell+}\beta_{\ell-}}+2.  \label{eq:s3}
\end{align}
With $\lambda$ and $s_\ell$ substituted, (\ref{eq:equiv-main2}) for $j=\ell$ can be re-written as one polynomial equation in $s_{\ell+},s_{\ell-}$, displaying $\mathcal{X}_3^{\rm reg}$ as an affine algebraic curve.
We may regard $t^2$ as a rational function on $\mathcal{X}_3^{\rm reg}$ through (\ref{eq:last}); moreover, when $s_{\ell+}\ne s_{\ell-}$, by Remark \ref{rmk:implication} (ii), the condition (\ref{eq:last}) is equivalent to
$$t^2=\frac{(s_{\ell+}+2)(\gamma_{\ell+}-\beta_{\ell+})^2-(s_{\ell-}+2)(\gamma_{\ell-}-\beta_{\ell-})^2}
{(\gamma_{\ell+}-\beta_{\ell+})^2-(\gamma_{\ell-}-\beta_{\ell-})^2}.$$

\section{The A-polynomial}

Choose a meridian-longitude pair $(\mathfrak{m},\mathfrak{l})$ of $K$, and denote the corresponding elements of $\pi_{1}(E_{K})$
by the same notations. Following \cite{LR03} (see Page 303), let $\mathcal{R}_{U}(K)$ denote the subset of $\mathcal{R}(K)$ consisting of representations which send $\mathfrak{m}$ and $\mathfrak{l}$ to upper-triangular matrices.
Define
$$\xi=(\xi_{1},\xi_{2}):\mathcal{R}_{U}(K)\to\mathbb{C}^{2}$$ by
setting $\xi_{1}(\rho)$ (resp. $\xi_{2}(\rho)$) to be the upper-left entry of $\rho(\mathfrak{m})$ (resp. $\rho(\mathfrak{l})$).
Taking the Zariski closure of the image of $\xi$, we obtain an affine algebraic curve $\mathcal{V}$, which is known to have the property that each component has dimension zero or one. The {\it A-polynomial} $A_K$ is defined to be the defining polynomial of the one-dimensional part of $\mathcal{V}$.

A-polynomial is notoriously difficult to compute. Till now, in the literary, formulas for A-polynomials have been obtained for few families of knots; see \cite{GM11,HL16,Ma14,Pe15} and the references therein.

Here we present a method for deriving a formula for the A-polynomial of an odd classical pretzel knot, leaving practical computations to possible future work.

For the pretzel knot $K$, take $\mathfrak{m}=x_1$. Let $\rho$ be a representation of $\pi_1(E_K)$ as in Section 3, but this time we require $\rho\in\mathcal{R}_U(K)$.

For $j=1,2,3$, let
\begin{align}
B_{j}&=Y_{j+}^{-k_{j+}}Y_{j-}^{k_{j-}+1} \\
&=-\beta_{j+}\gamma_{j-}Y_j^{-1}+\beta_{j+}\beta_{j-}Y_{j+}+\gamma_{j+}\gamma_{j-}Y_{j-}-\gamma_{j+}\beta_{j-}I. \label{eq:Bj}
\end{align}
Observe that
\begin{align}
B_{1}B_{2}B_{3}=I.   \label{eq:B}
\end{align}

Let $L_j$ denote the image under $\rho$ of the longitude associated to $x_j$ (so $L_1=\rho(\mathfrak{l})$).
Using (\ref{eq:zn}), we find
\begin{align}
L_j&=Y_{j-}^{-k_{j-}}Y_{j}^{k_{j}+1}Y_{j+}^{-k_{j+}}Y_{j-}^{k_{j-}+1}Y_{j}^{-k_{j}}Y_{j+}^{k_{j+}+1}=B_{j+}B_{j}B_{j-}. \label{eq:Lj}
\end{align}

The 1-dimensional part of $\mathcal{X}^{\rm irr}(K)$ consists of $\mathcal{X}_3$ and $\mathcal{X}_2$, the latter having $n=(k_1+1)(k_2+1)(k_3+1)$ components, so $A_K$ can be factorized as
\begin{align}
A_K=\overline{A}_K\cdot\prod\limits_{h=1}^nA_K^{(h)}.
\end{align}
Note that $s_1,s_2,s_3$ are all known constants on each component of $\mathcal{X}_2$, thus the corresponding factor $A_K^{(h)}$ is easy to understand, as explained in Remark \ref{rmk:trivial-part}.

The main achievement of this section is concerned with the ``hard" factor $\overline{A}_K$ contributed by $\mathcal{X}_3$.

Suppose the upper-left entries of $X_1,L_1$ are $u,w$, respectively, so that ${\rm tr}(L_j)=w+w^{-1}$, $j=1,2,3$.
\begin{prop}
If $\chi(\rho)\in\mathcal{X}_3$, then
\begin{align}
(1+w)(u+u^{-1})(\sigma_1+2-2\lambda)=(1-w)(u-u^{-1})(\sigma_1+2-2t^2).  \label{eq:u-w}
\end{align}
\end{prop}

\begin{proof}
It suffices to prove (\ref{eq:u-w}) when $\chi(\rho)$ belongs to an open subset of $\mathcal{X}_3$. Let us assume $u\notin\{\pm 1\}$ and $\lambda\ne 2+s_j, j=1,2$.

It is easy to see that
\begin{align}
L_j=\frac{w-w^{-1}}{u-u^{-1}}X_{j}+\frac{uw^{-1}-wu^{-1}}{u-u^{-1}}I.
\end{align}
Let $q={\rm tr}(B_3)$. By (\ref{eq:B}), ${\rm tr}(B_2B_1)={\rm tr}(B_3^{-1})=q$.
By (\ref{eq:Lj}), $B_2B_1=L_1B_3^{-1}=B_3^{-1}L_2$, hence
\begin{align}
q&={\rm tr}(L_1B_{3}^{-1})=\frac{w-w^{-1}}{u-u^{-1}}{\rm tr}(X_1B_{3}^{-1})+\frac{uw^{-1}-wu^{-1}}{u-u^{-1}}q,  \nonumber \\
q&={\rm tr}(L_2B_{3}^{-1})=\frac{w-w^{-1}}{u-u^{-1}}{\rm tr}(X_2B_{3}^{-1})+\frac{uw^{-1}-wu^{-1}}{u-u^{-1}}q.  \label{eq:t3-2}
\end{align}
As a consequence,
\begin{align}
{\rm tr}(X_1B_3^{-1})={\rm tr}(X_2B_3^{-1}). \label{eq:AP-key}
\end{align}

By (\ref{eq:Bj}),
\begin{align*}
B_3X_2^{-1}=-\beta_1\gamma_{2}X_2X_1^{-1}X_2^{-1}+\beta_1\beta_{2}X_2X_{3}^{-1}X_{2}^{-1}+\gamma_1\gamma_{2}X_3X_1^{-1}X_2^{-1}-
\gamma_1\beta_{2}X_{2}^{-1}, \\
\end{align*}
hence
\begin{align}
t^{-1}\cdot{\rm tr}(X_2B_{3}^{-1})&\ =-\beta_1\gamma_2+\beta_1\beta_2+\gamma_1\gamma_2(s_{1}+s_2+1-\lambda)-\gamma_1\beta_{2} \nonumber \\
&\ =(\sigma_1-\lambda-s_3)\gamma_1\gamma_2+(\gamma_1-\beta_1)(\gamma_2-\beta_2) \label{eq:longi1}  \\
&\stackrel{(\ref{eq:equiv-main2})}{=}\vartheta\left(\prod\limits_{j=1}^3(\sigma_1-\lambda-s_j)+(\sigma_1+2-2\lambda)^2\right) \nonumber \\
&\ =\vartheta((\sigma_1+2-\lambda)\kappa t^{-2}-\delta) \nonumber \\
&\stackrel{(\ref{eq:last})}{=}\vartheta(\sigma_1+2-\lambda-t^2)\kappa t^{-2},   \label{eq:tr-final}
\end{align}
where
\begin{align}
\vartheta=\frac{\beta_1\beta_2}{(\lambda-2-s_1)(\lambda-2-s_2)}.
\end{align}
By (\ref{eq:Bj}) again,
\begin{align*}
X_1B_{3}=-\beta_1\gamma_{2}X_1X_{2}X_1^{-1}+\beta_1\beta_{2}X_1X_{2}X_{3}^{-1}+\gamma_1\gamma_{2}X_1X_{3}X_1^{-1}-\gamma_1\beta_{2}X_1,
\end{align*}
hence
\begin{align}
t^{-1}\cdot{\rm tr}(X_1B_{3}^{-1})&\ =t^{-1}\cdot{\rm tr}(X_1(qI-B_{3}))=q-t^{-1}\cdot{\rm tr}(X_1B_{3}) \nonumber \\
&\ =q+\beta_{1}\gamma_2+(s_{3}+1-\lambda)\beta_1\beta_{2}-\gamma_1\gamma_{2}+\gamma_1\beta_{2} \nonumber \\
&\ =q+(s_{3}+2-\lambda)\beta_1\beta_{2}-(\gamma_1-\beta_1)(\gamma_2-\beta_{2}) \nonumber \\
&\stackrel{(\ref{eq:equiv-main2})}{=}q-\vartheta\left(\prod\limits_{j=1}^3(\lambda-2-s_j)+(\sigma_1+2-2\lambda)^2\right) \nonumber \\
&\ =q-\vartheta(\lambda\kappa t^{-2}-\delta)  \nonumber  \\
&\stackrel{(\ref{eq:last})}{=}q-\vartheta(\lambda-t^2)\kappa t^{-2}. 
\end{align}
Thus (\ref{eq:AP-key}) implies
\begin{align}
q=\vartheta(\sigma_1+2-2t^2)\kappa t^{-2}.  \label{eq:t3-3}
\end{align}
Combining (\ref{eq:t3-2}), (\ref{eq:tr-final}) and (\ref{eq:t3-3}), we obtain (\ref{eq:u-w}).
\end{proof}

Then $\overline{A}_K$ can be obtained by computing the multi-variable resultant of the following (remembering $t=u+u^{-1}$):
\begin{align}
(\mu-2-s_j)\gamma_j&=(\sigma_1-s_{j}-\mu)\beta_j, \qquad j=1,2,3, \label{eq:AP-1} \\
t^2(\lambda^2-(\sigma_1+2)\lambda+\sigma_2+4)&=4+\sigma_3+2\sigma_2-\sigma_1^2, \label{eq:AP-2}\\
(1+w)(u+u^{-1})(\sigma_1+2-2\lambda)&=(1-w)(u-u^{-1})(\sigma_1+2-2t^2). \label{eq:AP-3}
\end{align}

\begin{rmk} \label{rmk:trivial-part}
\rm If $\chi(\rho)$ lies in the $h$-th component of $\mathcal{X}_2$, then (\ref{eq:t3-2}), (\ref{eq:longi1}) and (\ref{eq:AP-2}) are still valid (whenever $t\notin\{0,\pm 2\}$). Eliminating $\lambda$ from these, we may obtain a polynomial in $u,w$, which is the very $A^{(h)}_K$.
\end{rmk}

\begin{exmp} 
\rm As an illustration, consider the case $k_1=k_2=1$. Then $\beta_1=\beta_2=1,\gamma_1=s_1,\gamma_2=s_2$.
From (\ref{eq:AP-1}) for $j=1,2$ we obtain
$$\lambda=s_1+s_2+1, \qquad s_3=s_1s_2+1,$$
then use them to re-write (\ref{eq:AP-2}), (\ref{eq:AP-3}) as, respectively,
\begin{align}
s_3&=(s_1+s_2)^2-(u^2+u^{-2}+2)(s_1+s_2)+2(u^2+u^{-2}+1), \label{eq:s3-quadratic} \\
s_1+s_2&=\frac{wu^2+1}{w+u^2}s_3-\frac{(w-1)(u^4-u^{-2})}{w+u^2}. \label{eq:s1+s2}
\end{align}
Use (\ref{eq:s1+s2}), we can re-write (\ref{eq:AP-1}) for $j=3$ as
\begin{align}
((w-1)(u^2-1)s_3+w(u^{-2}+1-u^4)+u^4+u^2-u^{-2})\gamma_3+(w+u^2)\beta_3=0.  \label{eq:s3-high}
\end{align}
Then $\overline{A}_K$ (for $K=P(3,3,2k_3+1)$) may be obtained from (\ref{eq:s3-quadratic})--(\ref{eq:s3-high}) by eliminating $s_1+s_2$ and $s_3$.
\end{exmp}

\end{document}